\documentclass{amsart}
\usepackage{amssymb, latexsym, amscd}
\usepackage[mathscr]{eucal}
\usepackage{enumerate}
\usepackage{fullpage}

\usepackage[pdftex]{graphicx, color}
\usepackage{sublabel}

\usepackage{ottmacros}

\usepackage[colorlinks=true,citecolor=blue]{hyperref}

\theoremstyle{plain}
\newtheorem{theorem}{Theorem}[section]

\newtheorem{lemma}[theorem]{Lemma}
\newtheorem{proposition}[theorem]{Proposition}

\newtheorem{mtheorem}{Theorem}

\theoremstyle{definition}
\newtheorem{definition}[theorem]{Definition}

\newtheorem{remark}[theorem]{Remark}

\newtheorem{example}[theorem]{Example}

\newtheorem*{notationx}{Notation}

\newcommand{\beq}{\begin{equation}}
\newcommand{\eeq}{\end{equation}}
\newcommand{\beqn}{\begin{equation*}}
\newcommand{\eeqn}{\end{equation*}}

\newcommand{\ie}{\textit{i.e.}}

\newcounter{trinity}

\numberwithin{equation}{section}

\begin{document}

\title{Memory loss for time-dependent dynamical systems}

\author{William Ott}
\address[William Ott]{
Courant Institute of Mathematical Sciences\\
New York, NY 10012, USA}
\urladdr{http://www.cims.nyu.edu/$\sim$ott}
\thanks{William Ott is partially supported by NSF postdoctoral fellowship DMS 0603509.}

\author{Mikko Stenlund}
\address[Mikko Stenlund]{
Courant Institute of Mathematical Sciences\\
New York, NY 10012, USA; Department of Mathematics and Statistics, P.O. Box 68, Fin-00014 University of Helsinki, Finland.}
\urladdr{http://www.math.helsinki.fi/mathphys/mikko.html}
\thanks{Mikko Stenlund is partially supported by a fellowship from the Academy of Finland.}

\author{Lai-Sang Young}
\address[Lai-Sang Young]{
Courant Institute of Mathematical Sciences\\
New York, NY 10012, USA}
\urladdr{http://www.cims.nyu.edu/$\sim$lsy}
\thanks{Lai-Sang Young is partially supported by NSF grant DMS 0600974.}

\keywords{memory loss, time-dependent dynamical systems, coupling, expanding maps,
  piecewise expanding maps}

\subjclass[2000]{37C60, 37C40}

\date{\today}

\begin{abstract} 
  This paper discusses the evolution of probability distributions for certain
  time-dependent dynamical systems. Exponential loss of memory is proved for expanding
  maps and for one-dimensional piecewise expanding maps with slowly varying parameters.
\end{abstract}

\maketitle

%%%%%%%%%%%%%%%%%%%%
%%%              %%%
%%% Introduction %%%
%%%              %%%
%%%%%%%%%%%%%%%%%%%%

\section{Introduction}\label{s:intro}

This paper is about statistical properties of nonautonomous dynamical systems, such as
flows defined by time-dependent vector fields or their discrete-time counterparts
described by compositions of the form $f_n \circ \cdots \circ f_2 \circ f_1$ where all the
$f_i: X \to X$ are self-maps of a space $X$. The topic to be discussed is the degree to
which such a system retains its memory of the past as it evolves with time.

Memory is lost when the initial state of a system is quickly forgotten. Conceptually, this
can happen in two very different ways.  The first is for trajectories to merge, so that in
time, they evolve effectively as a single trajectory independent of their points of
origin.  This happens in systems that are contractive. Consider for example a system
defined by the composition of a sequence of maps $f_i$ of a compact metric space $X$ to
itself, and assume that all the $f_i$ have a uniform Lipschitz constant $L<1$, {\it i.e.},
for all $x,y \in X$, $d(f_ix,f_iy) \leqs Ld(x,y)$. Since the diameter of the image of $X$
decreases exponentially with time, all trajectories eventually coalesce into an
exponentially small blob, which in general continues to evolve with time (except when all
the $f_i$ have the same fixed point).  A similar phenomenon is known to occur in random
dynamical systems. An SDE of the form
\begin{equation}\label{e:gen_sde}
dx_{t} = a(x_{t}) \, dt + \sum_{i=1}^{n} b_{i} (x_{t}) \circ dW_{t}^{i}
\end{equation}
gives rise to a stochastic flow of diffeomorphisms, in which almost every Brownian path
defines a time-dependent flow (see {\it e.g.}~\cite{Kh1997}).  When all of the Lyapunov
exponents are strictly negative, trajectories are known to coalesce into {\it random
  sinks} (see~\cite{Bp1992, LJy1985}).  This phenomenon occurs naturally in applications,
such as the Navier-Stokes system with sufficiently large viscosity (see {\it
  e.g.}~\cite{MnYls2002, Mj1999}), and in certain neural oscillator networks (see {\it
  e.g.}~\cite{LkSBeYls2008}).

In chaotic systems (autonomous or not), memory is lost quickly not through the coalescing
of trajectories but for a diametrically opposite reason, namely their sensitive dependence
on initial conditions.  Small errors multiply quickly with time, so that in practice it is
virtually impossible to track a specific trajectory in a chaotic system.  For this reason,
a statistical approach is often taken.  Let $\rho_0$ denote an initial probability density
with respect to a reference measure $m$, and suppose its time evolution is given by
$\rho_t$. As with individual trajectories, one may ask if these probability distributions
retain memories of their pasts.  We will say a system loses its memory in the statistical
sense if for two initial distributions $\rho_0$ and $\hat \rho_0$, $\int |\rho_t-\hat
\rho_t| \, dm \to 0$ as $t \to \infty$.  It is this form of memory loss that is studied in
the present paper. Of particular interest is when memory is lost quickly: we say a system
has {\it exponential statistical loss of memory} if there is a number $\alpha>0$ such that
for any $\rho_0$ and $\hat \rho_0$, $\int |\rho_t-\hat \rho_t| \, dm < C e^{-\alpha t}$.
Such memory loss may happen over a finite time interval, \ie, for $t \leqs T$, or for all
$t \geqs 0$.

Observe that while the two forms of memory loss described above are quite different on the
phenomenological level, the latter can be seen mathematically as a manifestation of the
first: By viewing $\{\rho_t\}_{t \geqs 0}$ as a trajectory in the space of probability
densities, statistical loss of memory is equivalent to $\rho_t$ and $\hat \rho_t$ having a
common future.  The results of this paper are based on this point of view.

\smallskip Before proceeding to specific results, we first describe a model that we think
is very useful to keep in mind, even though the analysis of this model is somewhat beyond
the scope of the present work.

\begin{example}
  {\bfseries Lorentz gas with slowly moving scatterers.}  The $2$-dimensional periodic
  Lorentz gas is usually modeled by the uniform motion of a particle in a domain $X =
  {\mathbb T}^2 \setminus \bigcup_i \Gamma_i$ where the $\Gamma_i$ are pairwise disjoint
  convex subsets of ${\mathbb T}^2$ and the particle bounces off the ``walls'' of this
  domain (equivalently the boundaries of the scatterers) according to the rule that the
  angle of incidence is equal to the angle of reflection.  In this model, the scatterers
  represent very heavy particles or ions, which move so slowly relative to the light
  particle (the one whose motion is described by the billiard flow) that one generally
  assumes they are fixed.  This is the traditional setup in billiard studies.  In reality,
  however, these large particles are bombarded by many light particles, and we focus on
  only one tagged light particle. The bombardments do cause the large particles to move
  about, though very slowly, and effectively independently of the motion of the tagged
  particle.  Thus one can argue that it is more realistic to model the situation as a
  billiard flow in a {\it slowly varying environment}, {\it i.e.}, where the positions of
  the scatterers change very slowly with time.  (See the recent work~\cite{CnDd2008},
  which attempts to model the motion of a single heavy particle.)
\end{example}

In this paper, we prove exponential loss of memory in the statistical sense discussed
above for time-dependent systems defined by expanding and piecewise expanding maps, the
latter in one dimension only. Expanding maps (time-dependent or not) provide the simplest
paradigms for exponential loss of memory in the statistical sense; we use them to
illustrate our ideas on the most basic level as their analysis requires few technical
considerations.  Piecewise expanding maps, on the other hand, begin to exhibit some of the
characteristics of the time-dependent billiard maps in the guiding example above.  Our
results can therefore be seen as a first step toward this physically relevant system.

The results of this paper apply to finite as well as infinite time, and our setting
extends not only that of iterations of single maps (for which results on correlation decay
for expanding maps and $1D$ piecewise expanding maps are not new), but it also includes
skew products in which fiber dynamics are of these types as well as random compositions.
What is different and new here is that the stationarity of the process is entirely
irrelevant.  Nor do the constituent maps have to belong to a bounded family, in which case
the rates of memory loss may vary accordingly.  A study which is closest to ours in spirit
is~\cite{LaYj1996}.

Coupling methods are used in this paper, although we could have used spectral arguments,
the Hilbert metric, or other techniques (see {\itshape e.g.}~\cite{BlSyCn1991, Cn2006,
  Lc1995a, Rd1989, Rd2004, Yls1998, Yls1999}).  We do not claim that our methods are
novel. On the contrary, one of the points of this paper is that under suitable conditions,
existing methods for autonomous systems can be adapted to give results for this
considerably broader class of dynamical settings, and we identify some of these
conditions.  Finally, even though coupling arguments have been used in more sophisticated
settings, see \textit{e.g.}~\cite{BxFrGa1999, BxLc2002, Cn2006, Yls1999}, we were unable
to locate a coupling-based proof for single expanding maps.  Section~\ref{s:uemaps} will
include this as a special case.

\begin{notationx}
The following notation is used
throughout: given $f_i: X \to X$ for $i \in \mbb{N}$,
\begin{enumerate}
\item
for $n \geqs m$, we write $F_{n,m}= f_n \circ \cdots \circ
f_m$;
\item
for $n \geqs 1$, we write $F_n = F_{n,1}$.
\end{enumerate}
\end{notationx}

%%%%%%%%%%%%%%%%%%%%%%%%%%%%%%%%
%%%                          %%%
%%% Uniformly expanding maps %%%
%%%                          %%%
%%%%%%%%%%%%%%%%%%%%%%%%%%%%%%%%

\section{Time-dependent expanding maps}\label{s:uemaps}

\subsection{Results}\label{ss:ue_stat_res}
Let $M$ be a compact, connected Riemannian manifold without boundary.  A smooth map
$f:M\to M$ is called {\it expanding} if there exists $\lambda>1$ such that
\begin{equation*}
|Df(x)v|\geqs \lambda |v|
\end{equation*}
for every $x \in M$ and every tangent vector $v$ at $x$.  Expanding maps provide the
simplest examples of systems with exponential loss of statistical memory.

First we introduce some frequently-used notation.  If $\nu$ is a Borel probability measure
on $M$, then we let $f_*\nu$ denote the measure obtained by transporting $\nu$ forward
using $f$, {\it i.e.}, $f_*\nu(E)= \nu(f^{-1}E)$ for all Borel sets $E$.  If $d\nu=\varphi
\, dm$ where $m$ is the Riemannian measure on $M$, then the density of $f_*\nu$ is given
by ${\mathcal P}_f(\varphi)$ where
\begin{equation*}
{\mathcal P}_f(\varphi)(x) \defas \sum_{y\in f^{-1}x}\frac{\varphi(y)}{|\det Df(y)|}.
\end{equation*}
Here ${\mathcal P}_f$ is the transfer operator associated with the map $f$; ${\mathcal
  P}_{F_n}$ is defined similarly.

In order to have a uniform rate of memory loss, we need
to impose some bounds on the set of mappings to be
composed. For $\lambda \geqs 0$ and $\Gamma \geqs 0$, define
\begin{equation*}
\mathcal {E}(\lambda, \Gamma)\defas\big\{f:M\to M : \| f \|_{\mathcal{C}^2} \leqs \Gamma,
\; \: |Df(x)v|\geqs \lambda |v| \; \: \forall \, (x,v)\big\}
\end{equation*}
and let 
\begin{equation*}
\mathcal{D}\defas \big\{ \varphi> 0 : \int\varphi \,dm=1, \; \: \text{$\varphi$ is Lipschitz}\big\}.
\end{equation*}

\begin{mtheorem}\label{mt:ue_maps}
  Given $\lambda$ and $\Gamma$ with $\lambda>1$, there exists a constant $\Lambda =
  \Lambda(\lambda, \Gamma)\in(0,1)$ such that for any sequence $f_i \in {\mathcal
    E}(\lambda, \Gamma)$ and any $\varphi, \psi \in {\mathcal D}$, there exists
  $C_{(\varphi,\psi)}$ such that
\begin{equation}\label{match}
\int |{\mathcal P}_{F_n}(\varphi)-{\mathcal P}_{F_n}(\psi)| \, dm\leqs C_{(\varphi,\psi)}
\Lambda^n \quad \forall \, n\geqs 0.
\end{equation}
\end{mtheorem}

\begin{remark}
We have assumed in Theorem~\ref{mt:ue_maps} that all of the $f_{i}$ are in a single
$\mcal{E} (\la, \Ga)$.  It will become clear that more general results in which $\la$ and
$\Ga$ are allowed to vary with $i$ can be formulated and proved by concatenating the
arguments below.
\end{remark}

\begin{remark}
  Correlation decay for expanding maps has been studied before.  For a single map, see
  {\itshape e.g.}~\cite{Rd1989, Rm1989}.  For random compositions, see {\itshape
    e.g.}~\cite{BvYls1993, BvYls1994}.  For time-dependent maps,~\cite{LaYj1996} proves
  that $\int |\mcal{P}_{F_{n}} (\vp) - \mcal{P}_{F_{n}} (\psi)| \, dm \to 0$ as $n \to
  \infty$ without discussing the rate of convergence.
\end{remark}

\subsection{Outline of proof}

Let $\ve >0$ be a small number to be determined, and fix $\lambda_0>1$ so that for
all $f \in {\mathcal E}={\mathcal E}(\lambda, \Gamma)$, we have $d(fx,fy) \geqs \lambda_0
d(x,y)$ whenever $d(x,y)< \ve$.  Here $d(\cdot,\cdot)$ denotes Riemannian distance.
For $L>0$, we define
\begin{equation*}
  \mathcal{D}_L\defas \left\{\varphi>0 : \int\varphi \,dm=1,\; \:
    \left|\frac{\varphi(x)}{\varphi(y)}-1\right | \leqs L d(x,y) \; \,\text{if} \; \, d(x,y)<\ve \right\}.
\end{equation*}
Notice that ${\mathcal D} = \bigcup_{L >0} {\mathcal D}_L$: For $\varphi \in {\mathcal
  D}$,
\begin{equation*}
\left|\frac{\varphi (x)}{\varphi (y)} - 1 \right|
= \frac{1}{\varphi (y)} |\varphi (x)-\varphi (y)|
\leqs \frac{{\rm Lip}(\varphi)}{\min(\varphi)} d(x,y);
\end{equation*}
functions in ${\mathcal D}_L$ are clearly locally Lipschitz. Key to the proof is the
following observation:

\begin{proposition}\label{lem:smoothed}
  There exists $L^{*} > 0$ for which the following holds.  For any $L>0$, there exists
  $\tau(L) \in {\mathbb Z}^+$ such that for all $\varphi \in {\mathcal D}_L$ and $f_i \in
  {\mathcal E}$, ${\mathcal P}_{F_n}(\varphi)\in\mathcal{D}_{L^{*}}$ for all $n \geqs
  \tau(L)$.
\end{proposition}

As our proof in Section~\ref{ss:ue_pf_detls} will show, the choice of $L^*$ is arbitrary,
provided it is greater than a number determined by $\lambda$ and $\Gamma$.

Now let $f_i \in {\mathcal E}$ and $\varphi, \psi \in {\mathcal D}$ be given.  Then there
exists $N_0=N_0(\varphi, \psi)$ such that both ${\mathcal P}_{F_{N_0}}(\varphi)$ and
${\mathcal P}_ {F_{N_0}}(\psi)$ are in ${\mathcal D}_{L^*}$.  This waiting period is the
reason for the prefactor $C_{(\varphi, \psi)}$ on the right side of~\eqref{match}.  With
this out of the way, we may assume we start with two densities $\varphi, \psi \in
{\mathcal D}_{L^*}$ from here on.

Notice that all functions in ${\mathcal D}_{L^*}$ are $\geqs \kappa$ for some constant
$\kappa>0$; it is easy to see from the definition of ${\mathcal D}_{L^*}$ that they have
uniform lower bounds on $\ve$-disks.  We think of the measures $\varphi \, dm$ and
$\psi \, dm$ as having a part, namely $\kappa \, dm$, in common.  Since $(F_n)_*(\kappa \,
dm)$ will also be common to both $(F_n)_*(\varphi \, dm)$ and $(F_n)_*(\psi \, dm)$, we
regard this part of the two measures as having been ``matched''.  In order to retain
control of distortion bounds, however, we will ``match'' only half of what is permitted,
and renormalize the ``unmatched part'' as follows: Let
\begin{equation}
\label{renorm}
\hat \varphi = \frac{\varphi - \frac12 \kappa}
{1- \frac12 \kappa \cdot m(M)} \quad \text{and} \quad
\hat \psi = \frac{\psi - \frac12 \kappa}
{1- \frac12 \kappa \cdot m(M)}\ .
\end{equation}

\begin{lemma}\label{l:match_smooth}
  For $\varphi\in\mathcal{D}_{L^{*}}$, if $\hat \varphi$ is as above, then $\hat \varphi
  \in \mathcal{D}_{2L^{*}}$.
\end{lemma}

Let $N=\tau(2L^*)$ be given by Proposition~\ref{lem:smoothed}.  Then $\bar \varphi_N
\defas {\mathcal P}_{F_N} (\hat \varphi)$ and $\bar \psi_N \defas {\mathcal P}_{F_N}(\hat
\psi)$ are in ${\mathcal D}_{L^*}$. We subtract off $\frac12 \kappa$ from each of $\bar
\varphi_N$ and $\bar \psi_N$ and renormalize as in~\eqref{renorm}, obtaining $\hat
\varphi_N$ and $\hat \psi_N$ respectively. By Lemma~\ref{l:match_smooth}, they are in
${\mathcal D}_{2L^*}$.  In general, given $\hat \varphi_{(k-1)N}, \hat \psi_{(k-1)N} \in
{\mathcal D}_{2L^*}$, we let
\begin{equation*}
\bar \varphi_{kN} \defas {\mathcal P}_{F_{kN,(k-1)N+1}}
(\hat \varphi_{(k-1)N}) \quad \text{and} \quad \bar \psi_{kN}\defas 
{\mathcal P}_{F_{kN,(k-1)N+1}} (\hat \psi_{(k-1)N}). 
\end{equation*}
By Proposition~\ref{lem:smoothed}, $\bar \varphi_{kN}, \bar \psi_{kN} \in {\mathcal
  D}_{L^*}$. We subtract off $\frac12 \kappa$ and renormalize to obtain $\hat
\varphi_{kN}$ and $\hat \psi_{kN}$ in ${\mathcal D}_{2L^*}$ (Lemma~\ref{l:match_smooth}),
completing the induction.

Since a fraction of $\frac12 \kappa \cdot 
m(M)$ of the not-yet-matched parts of the measures
is matched every $N$ steps, we obtain
\begin{equation*}
\int |{\mathcal P}_{F_n}(\varphi) - 
{\mathcal P}_{F_n}(\psi)| \, dm \leqs  
2(1 - \frac{1}{2} \kappa \cdot m(M))^k \quad \text{for} \quad
kN \leqs n < (k+1)N.
\end{equation*}
This leads directly to the asserted exponential estimate.  \hfill $\blacksquare$
 
\begin{remark}
  Theorem~\ref{mt:ue_maps} also holds for initial densities that are not strictly positive
  provided one is able to guarantee that they eventually evolve into densities that are
  strictly positive. One way to make this happen is to have sufficiently many of the
  initial $f_i$ remain in a small enough neighborhood of some fixed $f \in {\mathcal E}$,
  and take advantage of the fact that every expanding map $f$ has the property that given
  any open set $U \subset M$, there exists $N(U) \in \mbb{N}$ such that $f^{n} (U) \supset
  M$ for all $n \geqs N(U)$.
\end{remark}
%%%%%%%%%%%%%%%%%%%%%%%%%%%%%%%%%%%%%%%%%%%%%

\subsection{Details of proof}\label{ss:ue_pf_detls}

We begin with an essential distortion estimate.

\begin{lemma}\label{lem:distortion}
  There exists a constant $C_0$ depending on $\lambda_{0}$ and $\Gamma$ such that
\begin{equation*}
\frac{|\det DF_n(x)|}{|\det DF_n(y)|} \leqs e^{C_0 d(F_n(x),F_n(y))}
\end{equation*}
for all $x, y \in M$ and $n \in \mbb{Z}^{+}$ with the property that
$d(F_k(x),F_k(y))<\ve$ for all $k < n$.
\end{lemma}

\begin{proof}[Proof of Lemma~\ref{lem:distortion}]
We have
\begin{align*}
\log \frac{|\det DF_n(x)|}{|\det DF_n(y)|} &= \sum_{k=1}^n \left( \log |\det Df_k(F_{k-1}(x))|  -  \log |\det Df_k(F_{k-1}(y))|\right) \\
&\leqs \sum_{k=0}^{n-1} C_1 d(F_{k}(x), F_{k}(y)) \leqs \sum_{k=0}^{n-1} C_1 \lambda_0^{-(n-k)} d(F_{n}(x), F_{n}(y))\\
&\leqs \frac{C_1}{\lambda_0-1}\,d(F_n(x),F_n(y)),
\end{align*}
where $C_1$ is an upper bound on the Lipschitz constant of the $\mathcal{C}^1$ function
$\log|\det Df|$ for any function $f$ in the family $\mathcal{E}$.
\end{proof}

We are in position to prove Proposition~\ref{lem:smoothed}, which asserts the existence of
$L^{*} > 0$ such that $\mathcal{D}_{L^{*}}$ attracts densities.

\begin{proof}[Proof of Proposition~\ref{lem:smoothed}]
  Let $y \in D(x,\ve)$ where $D(x,\ve)$ is the disk of radius $\ve$
  centered at $x$. We let $G_{n,i}$ be the $i^{\text{th}}$ branch of
  $F_n^{-1} | D(x,\ve)$, and let
\begin{equation*}
\varphi_n^i\defas  \frac{\varphi\circ G_{n,i}}{|\det DF_n\circ G_{n,i}|}.
\end{equation*}
Then $\varphi^i_n$ is the contribution to the density $\varphi_n\defas {\mathcal
  P}_{F_n}(\varphi) =\sum_i \varphi_n^i$ obtained by pushing along the $i^{\text{th}}$
branch. Estimating distortion one branch at a time, we have
\begin{equation*}
\frac{\varphi_n^i(x)}{\varphi_n^i(y)} = \left(\frac{\varphi(G_{n,i}(x))}{\varphi(G_{n,i}(y))}\right) \cdot 
\left(\frac{|\det DF_n(G_{n,i}(y))|}{|\det DF_n(G_{n,i}(x))|}\right). 
\end{equation*}
To estimate the first factor on the right, we use $d(G_{n,i}(x),G_{n,i}(y)) <
\lambda_0^{-n} d(x,y)$ and $\varphi\in\mathcal{D}_{L}$. To estimate the second factor, we
use Lemma~\ref{lem:distortion}.  Combining the two, we obtain
\begin{equation*}
\left|\log \frac{\varphi_n^i(x)}{\varphi_n^i(y)}\right| \leqs (L \lambda_0^{-n} + C_0) d(x,y).
\end{equation*}
Exponentiating, moving $\varphi^i_n(y)$ to the right side, and summing over $i$ before
dividing by $\vp_n$ again, we obtain
\begin{equation*}
\frac{\varphi_n(x)}{\varphi_n(y)} \leqs e^{(L \lambda_0^{-n}+C_0)d(x,y)}.
\end{equation*}
By taking $\ve$ small enough, we may assume
\begin{equation}\label{e:ly-expanding}
\left| \frac{\varphi_n(x)}{\varphi_n(y)} -1 \right| \leqs
 2 \left|\log \frac{\varphi_n(x)}{\varphi_n(y)}\right|
 \leqs 2 (L \lambda_{0}^{-n} + C_0) d(x,y).
\end{equation}
Finally, we choose $\ta (L)$ large enough so that $L \lambda_0^{-\ta (L)} \leqs C_0$, and
conclude that
\begin{equation*}
\left| \frac{\varphi_n(x)}{\varphi_n(y)} -1 \right| \leqs 
L^{*} d(x,y) 
\end{equation*}
for all $n \geqs \ta (L)$, where $L^{*} = 4 C_{0}$.
\end{proof}

Only the proof of Lemma~\ref{l:match_smooth} remains.

\begin{proof}[Proof of Lemma~\ref{l:match_smooth}]
The distortion of $\hat \varphi$ satisfies
\begin{equation*}
  \left| \frac{\hat \varphi(x)}{\hat \varphi(y)}-1 \right| = \left| \frac{\varphi(x)-\frac12
    \kappa}{\varphi(y)-\frac12 \kappa}-1 \right| =
\left| \frac{\frac{\varphi(x)}{\varphi(y)}-\frac{\tfrac{1}{2}
    \kappa}{\varphi(y)}}{1-\frac{\tfrac{1}{2} \kappa}{\varphi(y)}}-1 \right| = 
\frac{\left| \frac{\varphi(x)}{\varphi(y)}-1 \right|}{1-\frac{\tfrac{1}{2} \kappa}{\varphi(y)}}.
\end{equation*}
Since $\varphi\geqs \kappa$, the rightmost quantity above is $\leqs 2 \left|
  \frac{\varphi(x)}{\varphi(y)}-1 \right|$.  We conclude that $\hat \varphi \in {\mathcal
  D}_{2L^{*}}$ if $\varphi \in {\mathcal D}_{L^*}$.
\end{proof}

The proof of Theorem~\ref{mt:ue_maps} is now complete.

%%%%%%%%%%%%%%%%%%%%%%%%%%%%%%%%
%%%                          %%%
%%% Piecewise-expanding maps %%%
%%%                          %%%
%%%%%%%%%%%%%%%%%%%%%%%%%%%%%%%%

\section{Time-dependent $1D$ piecewise expanding maps }\label{s:pw_expand}

\subsection{Statement of results}\label{ss:results_pw}

We consider in this section piecewise $\mcal{C}^{2}$ expanding maps of the circle. More
precisely, we let ${\mcal S}^1$ be the interval $[0,1]$ with end points identified, and
say $f : \mcal{S}^{1} \to \mcal{S}^{1}$ is {\it piecewise $\mcal{C}^{2}$ expanding} if
there exists a finite partition $\mcal{A}_{1} = \mcal{A}_{1} (f)$ of $\mcal{S}^{1}$ into
intervals such that for every $I \in \mcal{A}_{1}$,
\begin{enumerate}
\item 
$f|I$ extends to a $\mcal{C}^{2}$ mapping in 
a neighborhood of $I$;
\item 
there exists $\lambda >1$ such that
$|f'(x)| \geqs \lambda$ for all $x \in I$.
\end{enumerate}
It simplifies the analysis slightly to assume $\lambda >2$, and we will do that (if $\la
\leqs 2$, we replace $f$ by a suitable power of $f$ and adjust the assumptions below
accordingly).

Unlike the case of expanding maps (with no discontinuities), compositions of piecewise
expanding maps do not necessarily have exponential loss of memory. Indeed, systems defined
by a single piecewise expanding map may not even be ergodic, and decay of correlations
(loss of memory) in that context is equivalent to mixing. Some additional conditions are
therefore needed for results along the lines of Theorem~\ref{mt:ue_maps}.  Let
$\mcal{A}_{n} \defas \bigvee_{i=1}^{n} f^{-(i-1)} \mcal{A}_{1}$ be the join of the
pullbacks of the partition $\mcal{A}_{1}$ and let $\mcal{A}_{n} | I$ be the restriction of
$\mcal{A}_{n}$ to the set $I$.  For $J \subset \mcal{S}^1$, let ${\rm int}(J)$ denote the
interior of $J$.

\begin{definition} 
  We say $f$ is {\it enveloping} if there exists $N \in {\mathbb Z}^+$ such that for every
  $I \in \mcal{A}_{1}$, we have
\begin{equation*}
\bigcup_{J \in \mcal{A}_N|I} f^N({\rm int}(J)) = \mcal{S}^1.
\end{equation*}
The smallest such $N$ is called the {\it enveloping time}.
\end{definition}

If the enveloping time of $f$ is $N$, then starting from any $I \in \mcal{A}_{1}$, $f^N|I$
{\it overcovers} $\mcal{S}^1$, in the sense that every $z \in \mcal{S}^1$ lies in $f^N(J)$
for some $J \in \mcal{A}_N|I$, and more than that: it is a positive distance from
$f^N(\partial J)$.  From here on, our universe $\mcal E$ is comprised of piecewise
$\mcal{C}^{2}$ expanding, enveloping maps.

For the same reason that many (individual) piecewise expanding maps are not mixing, one
cannot expect the arbitrary composition of piecewise expanding maps to produce exponential
loss of memory -- even when the constituent maps have good mixing properties: this is
because such properties do not necessarily manifest themselves in a single step. To
effectively leverage the mixing properties of individual maps, we may need a number of
consecutive $f_i$ to be near a single map.  We will formulate two sets of results: {\it a
  local result}, which assumes that all the $f_i$ are near a single piecewise expanding
map $g$, and {\it a global result}, which allows the $f_i$ to wander far and wide but
slowly.

\subsubsection{Local result}

Let $g \in \mcal E$ be fixed. We let $\Omega(g)=\{ x_{1} = x_{k+1}, x_{2}, \ldots, x_{k}
\} \subset \mcal{S}^1$ be the set of discontinuity points of $g$ labeled counterclockwise,
and let $d_\Omega(g) \defas \min_i |x_{i+1}-x_i|$.  For $\ve < \frac14 d_\Omega(g)$, we
say $f \in \mcal E$ is $\ve$-near $g$, written $f \in \mcal{U}_\ve (g)$, if the following
hold:
\begin{enumerate}[(1)]
\item 
$\Omega(f) = \{y_{1} = y_{k+1}, y_{2}, \ldots, y_{k} \}$
where $|y_i-x_i|< \ve$;
\item 
if $\xi_{fg}$ maps each interval $[x_i, x_{i+1}]$ 
affinely onto $[y_i,y_{i+1}]$, then  on each
$[x_i, x_{i+1}]$, 
\begin{equation*}
\|f\circ \xi_{fg}-g\|_{\mathcal{C}^2} < \ve\ .
\end{equation*}
\end{enumerate}

As in the case of single $1D$ piecewise expanding maps, a natural class of densities to
consider is
\begin{equation*}
\mcal{D} = \left\{ \vp \in \bv (\mcal{S}^{1}, \mbb{R}) : \vp \geqs 0, \; \: \int_{\mcal{S}^1}
\vp  (x) \, dx = 1 \right\}.
\end{equation*}

Recall the definitions of $F_n$ and $\mcal{P}_{F_n}$ from the end of Section~\ref{s:intro}
and the beginning of Section~\ref{s:uemaps}, respectively.

\begin{mtheorem}\label{mt:pw_local}
  Let $g \in \mcal{E}$.  There exist $\La < 1$ and $\ve > 0$ sufficiently small (depending
  on $g$) such that for all $f_i \in \mcal{U}_{\ve} (g)$ and $\vp, \psi \in \mcal{D}$,
  there exists $C_{(\vp, \psi)} > 0$ such that for all $n \in \mbb{Z}^{+}$, we have
\begin{equation}\label{e:m_locres_est}
\int_{\mcal{S}^{1}} \big|\mcal{P}_{F_n}(\varphi) - 
\mcal{P}_{F_n}(\psi)\big| \, dx \leqs
C_{(\vp, \psi)} \La^n .
\end{equation}
\end{mtheorem}

There exists an extensive literature on correlation decay for $1D$ piecewise expanding
maps in the contexts of a single map and random i.i.d.\@ compositions.  See,
{\textit{e.g.}},~\cite{BvYls1993, BvYls1994, HfKg1982, Lc1995}.

\subsubsection{Global result}

It is straightforward to verify that the collection of sets $\mcal{S} \defas
\{\mcal{U}_\ve(f) : f \in \mcal{E}, \; \: \ve < \frac{1}{4} d_\Omega(f)\}$ generates a
topology on $\mcal{E}$.\footnote{To prove $\mcal{S}$ forms the basis of a topology, it
  suffices to check that for $f_1, f_2 \in \mcal{E}$, $\ve_1, \ve_2>0$, and $g \in
  \mcal{U}_{\ve_1}(f_1) \cap \mcal{U}_{\ve_2}(f_2)$, there exists $\ve>0$ such that
  $\mcal{U}_{\ve}(g) \subset \mcal{U}_{\ve_1}(f_1) \cap \mcal{U}_{\ve_2}(f_2)$.}  Consider
now a continuous map $\ga : [a,b] \to \mcal E$ (see Figure~\ref{fi:curve}) and a finite or
infinite sequence of $f_i$ of the form $f_i=\gamma(t_i)$ where $a \leqs t_1 \leqs t_2
\leqs t_3 \leqs \cdots \leqs b$. Let $\Delta \defas \max_i (t_{i+1}-t_i)$. If we think of
the closed interval $[a,b]$ as time, then decreasing $\Delta$ corresponds to decreasing
the \textquoteleft velocity' at which the curve $\ga ([a,b])$ is traversed.

\begin{center}
\begin{figure}[ht]
\includegraphics[scale=0.75]{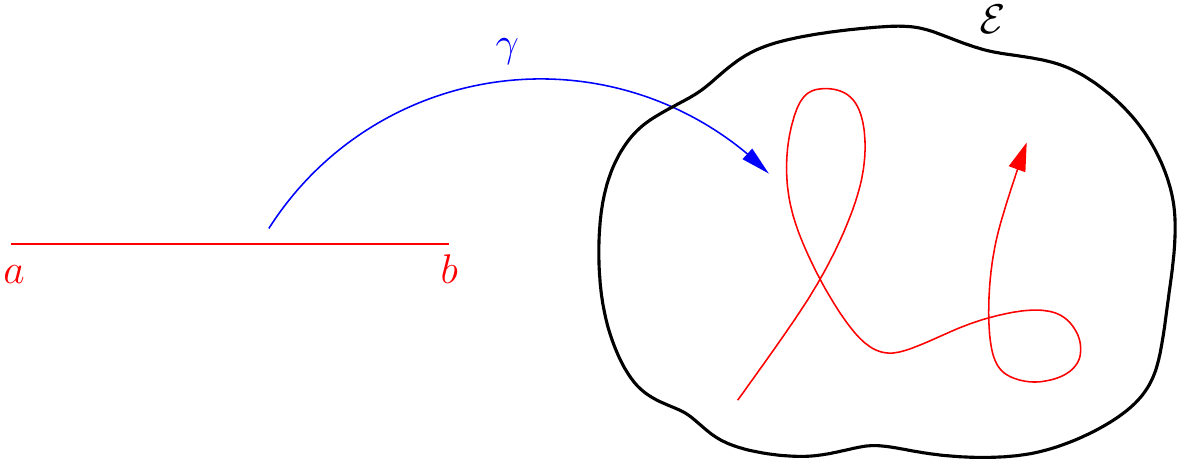}
\caption{The picture we envision is that of ``driving'' the system along a curve $\gamma$
  in $\mcal{E}$ and losing memory of past density distributions at {\it variable rates}
  depending on local characteristics.  That, we submit, is the true nature of {\itshape memory
    loss in dynamical systems with slowly varying parameters}.}
\label{fi:curve}
\end{figure}
\end{center}

\begin{mtheorem}\label{mt:pw_global}
  Let $\ga : [a,b] \to \mcal{E}$ be a continuous map.  Then there exist $\de_{0} > 0$ and
  $\La < 1$ (depending on $\gamma$) for which the following holds: For every $\{t_i\}$ as
  above with $\Delta \leqs \delta_0$ and $\vp, \psi \in \mcal{D}$, there exists $C_{(\vp,
    \psi)} > 0$ such that for all relevant $n$,
\begin{equation*}
\int_{\mcal{S}^{1}} |\mcal{P}_{F_n} (\vp) - 
\mcal{P}_{F_n}(\psi)| \, dx
\leqs C_{(\vp, \psi)} \La^n.
\end{equation*}
\end{mtheorem}

\begin{remark}\label{rmk:variable}
  We have tried not to overburden the formulation of Theorem~\ref{mt:pw_global}, but as
  will be clear from the proofs, various generalizations are possible: The curve can be
  defined on an infinite interval and can traverse various subregions of $\mcal{E}$ with
  nonuniform derivative bounds, leading to variable rates of memory loss.  One does not,
  in fact, have to start with a prespecified curve and occasional long distance jumps can
  be accommodated.
\end{remark}

\begin{remark} Finally, we note that Theorem~\ref{mt:pw_global} -- together with its
  generalizations mentioned in Remark~\ref{rmk:variable} -- is a simplified version of the
  Lorentz gas example in Section~\ref{s:intro}, an important difference being the absence
  of the stable directions.
\end{remark}

%%%%%%%%%%%%%%%%%%%%%%%%%%%%%%%

\subsection{Proof of local result}\label{ss:pf_loc_res}

The following is an outline of the main steps of our proof:

\smallskip
\refstepcounter{trinity}
\label{step1}
\noindent
{\itshape Step~\ref{step1}.} As in the expanding case (in Section~\ref{s:uemaps}), we
represent the set of densities $\mcal{D}$ as $\mcal{D}= \bigcup_a \mcal{D}_a$ where the
conditions on $\mcal{D}_a$ are more relaxed as $a$ increases, and show that there is an
$a^*$ for which $\mcal{D}_{a^*}$ is an attracting set under $\mcal{P}_{F_n}$ for any
sequence of $f_i$ in a subset of $\mcal E$ with uniform bounds.  The time it takes to
enter $\mcal{D}_{a^*}$ from each $\mcal{D}_a$ is shown to be bounded.

\smallskip
\refstepcounter{trinity}
\label{step2}
\noindent
{\itshape Step~\ref{step2}.} Unlike the expanding case, where all functions in this
attracting set are uniformly bounded away from $0$, and coupling (or matching of
densities) can be done immediately, we do not have such a bound here. Instead, we
guarantee the matching of a fixed fraction of the measures {\it a finite number of steps
  later} using the enveloping property of $g$.

\smallskip
\refstepcounter{trinity}
\label{step3}
\noindent
{\itshape Step~\ref{step3}.} To complete the cycle, we must show that after subtracting
off the amount matched and renormalizing as in Section~\ref{s:uemaps}, functions in
$\mcal{D}_{a^*}$ are in $\mcal{D}_a$ for some bounded $a$.

\smallskip

We now carry out these steps in detail.

\medskip
\noindent {\bfseries Step~\ref{step1}.} For $\varphi \in {\rm BV}
(\mcal{S}^1, {\mathbb R})$, we let $\bigvee_0^1 \varphi$
denote the total variation of $\varphi$, and let
\begin{equation*}
\mcal{D}_{a} \defas \left\{ \vp \in \bv (\mcal{S}^{1}, \mbb{R}) : \vp \geqs 0, \; \: \int \vp = 1, \; \: 
\bigvee_{0}^{1} \vp \leqs a \right\}.
\end{equation*}
Clearly, $\mcal{D}= \bigcup_a \mcal{D}_a$. Let
\begin{equation*}
\lambda (f) \defas 
\min_{I \in \mcal{A}_{1} (f)} \inf_{x \in I} |f'(x)|\ ,
\end{equation*}
and recall the following well-known inequality originally due to Lasota and Yorke.

\begin{lemma}[\textbf{Lasota-Yorke inequality}~\cite{LaYj1973}]\label{l:lyi}
  Let $f$ be a piecewise $\mcal{C}^{2}$ expanding map.  For $\vp \in \bv (\mcal{S}^{1},
  \mbb{R})$, we have
\begin{equation}\label{e:lyi}
\bigvee_{0}^{1} \mcal{P}_f (\vp) \leqs 2 \la (f)^{-1} 
\bigvee_{0}^{1} \vp + A(f) \| \vp
\|_{L^{1}}
\end{equation}
where
\begin{equation*}
A(f) = \sup_{z \in \mcal{S}^{1}} \frac{|f''(z)|}{|f'(z)|^{2}} + 2 \sup_{I \in \mcal{A}_{1}
  (f)} \frac{\sup_{z \in I} |f'(z)|^{-1}}{|I|}.
\end{equation*}
\end{lemma}

We now fix $\mcal{E}_0 \subset \mcal E$ with uniform $\mcal{C}^{2}$ bounds and with $g$
well inside $\mcal{E}_0$.  Let 
\begin{equation*}
\lambda_0 \defas \inf_{f \in \mcal{E}_0} 
\lambda(f)
\quad \text{and} \quad A_0 \defas
\sup_{f \in \mcal{E}_0} A(f).
\end{equation*}
We assume $\la_{0} > 2$.  Upon repeated applications of~\eqref{e:lyi}, for $f_i \in
\mcal{E}_0$ and $\varphi\in\mathcal{D}$ we obtain
\begin{equation}
\label{e:genlyi}
%\begin{aligned}
\bigvee_{0}^{1} \mcal{P}_{F_n}(\vp) 
\leqs  
%(2 \la_0^{-1})^{n} \bigvee_{0}^{1} \vp + A_0  \sum_{i=0}^{n-1} (2 \la_{0}^{-1})^{i}
% < 
(2 \la_0^{-1})^{n} \bigvee_{0}^{1} \vp + 
\frac{A_0}{1 - 2 \la_0^{-1}} ,
%\end{aligned}
\end{equation}
which is the analog of the distortion estimate~\eqref{e:ly-expanding} in Section~\ref{s:uemaps}.

Our main result in Step~\ref{step1} is

\begin{proposition}\label{p:pw_absorb}
  Fix any $a^* > \frac{A_0}{1- 2\lambda_0^{-1}}$. Then for every $a>0$, there exists
  $\tau(a) \in {\mathbb Z}^+$ such that for all $f_i \in \mcal{E}_0$, $\vp \in \mcal{D}_a$
  and $n \geqs \tau(a)$, $\mcal{P}_{F_n}(\vp) \in \mcal{D}_{a^*}$.
\end{proposition}

\begin{proof} 
  This is an immediate consequence of~\eqref{e:genlyi}.  In fact, it is enough to choose
\begin{equation}\label{e:absorb_est}
  \tau(a) \geqs   
   \ln \left( \left( a^* - \frac{A_0}
        {1 - 2 \la_0^{-1}} \right)a^{-1} \right) \bigg/  \ln (2 \la_0^{-1})
%   \tau(a) \leqs \max \left\{ 0, 1 +  
%    \big( \ln (2 \la_0^{-1}) \big)^{-1} \ln \left( \left( a^* - \frac{A_0}
%        {1 - 2 \la_0^{-1}} \right)a^{-1} \right) 
%  \right\}.
\end{equation}
among nonnegative integers.
\end{proof}

\noindent {\bfseries Step~\ref{step2}.} The second step is perturbative. We will first
work with iterates of $g$ before extending our results to $f_{i}$ in some suitable
$\mcal{U}_\ve(g)$.

\begin{lemma}\label{l:single_pos}
  There exist $n_0 \in {\mathbb Z}^+$ and $\kappa_0>0$ (depending on $g$) such that for
  all $\vp \in \mcal{D}_{a^*}$, $\mcal{P}_{g^{n_0}}(\varphi) \geqs \kappa_0$.
\end{lemma}

\begin{proof} 
  Let $\mcal{A}_{1}$ be the partition for $g$, and let $n_1$ be such that all elements of
  $\mcal{A}_{n_1}$ have length $<\frac{1}{2a^*}$.  We will show that for every $\vp \in
  \mcal{D}_{a^*}$ there exists $J =J(\varphi) \in \mcal{A}_{n_1}$ such that $\vp | J \geqs
  \frac12$. Suppose, to derive a contradiction, that for each $J \in \mcal{A}_{n_1}$,
  there exists $z_J \in J$ with $\vp(z_J)<\frac12$. Then
\begin{equation*}
\int_J \vp \leqs |J| \left(\vp(z_J)+\bigvee_J \vp 
\right) < \frac{|J|}{2} +
\frac{1}{2a^*} \bigvee_J \vp.
\end{equation*}
Summing over $J$, this gives $\int_{\mcal{S}^1} \vp < \frac{1}{2} + \frac{1}{2} = 1$.

Next we claim that for every $J \in \mcal{A}_{n_{1}} (g)$, there exists $s=s(J)$ and a
subinterval $J_{s} \subset J$ such that $g^{s} | J_{s}$ is $\mcal{C}^{2}$ and $g^{s}
(J_{s}) \supset I$ for some $I \in \mcal{A}_{1} (g)$.  To prove this, we inductively
define a nested sequence of intervals $J = J_{1} \supset J_{2} \supset J_{3} \supset
\cdots$ as follows.  Assume that $J_{k}$ has been defined.  If $g^{k} (J_{k}) \supset I$
for some $I \in \mcal{A}_{1} (g)$, set $s = k$.  If not, then either $g^{k} (J_{k})
\subset I$ for some $I \in \mcal{A}_{1} (g)$ or $g^{k} (J_{k})$ intersects $2$ elements of
$\mcal{A}_{1} (g)$.  In the former case, set $J_{k+1} = J_{k}$, and in the latter, let
$J_{k+1}$ be the longer of the $2$ intervals in $\mcal{A}_{1} (g) | g^{k} (J_{k})$.  This
process must terminate in a finite number of steps because $\inf |g'| > 2$.

Let $n_0 \defas s_0+N$ where $s_0=\max\{s(J): J \in \mcal{A}_{n_1}\}$ and $N=N(g)$ is the
enveloping time for $g$.  We now produce the $\kappa_0$ with the asserted property in the
lemma.  Fix arbitrary $\vp \in \mcal{D}_{a^*}$. Let $J=J(\vp) \in \mcal{A}_{n_1}$ be such
that $\vp | J \geqs \frac12$, and let $I \in \mcal{A}_{1}$ be such that $g^{s(J)}(J_{s})
\supset I$.  Then $\mcal{P}_{g^{s(J)}}(\vp)| I \geqs \frac{1}{2} M_{0}^{-s(J)}$ where
$M_{0} (g) \defas \sup |g'|$.  From $g^{N}(I)=\mcal{S}^1$, it follows that
$\mcal{P}_{g^{s(J)+N}}(\vp) \geqs \frac{1}{2} M_{0}^{-(s(J)+N)}$ on $\mcal{S}^1$. We still
have some steps to go if $s(J)<s_0$, but $g$ is onto (as all enveloping maps are
necessarily onto), and even in the worst-case scenario, we still have
$\mcal{P}_{g^{n_0}}(\vp) \geqs \frac{1}{2} M_{0}^{-n_0} \defas \kappa_0$ everywhere on
$\mcal{S}^1$.
\end{proof}

Define
\begin{equation*}
\mcal{A} (F_n) \defas \bigvee_{i=1}^{n} (F_{i-1})^{-1} \mcal{A}_{1} (f_{i})
\end{equation*}
where $F_{0}$ is the identity map.  Now let $f_i \in \mcal{U}_{\ve}(g)$.  From the
one-to-one correspondence between elements of $\mcal{A}_{1} (f_{i})$ and $\mcal{A}_{1}
(g)$, one deduces that provided $\ve$ is sufficiently small, there is a well-defined
mapping $\Phi_{n} : \mcal{A}_{n} (g) \to \mcal{A}(F_n)$ where for $J \in \mcal{A}_{n} (g)$,
$\Phi_{n} (J) \in \mcal{A}(F_n)$ has the same itinerary as $J$. (In general, $\Phi_{n}$ need not be
onto.)  For $J=(a,b)$, let $J_\delta \defas (a+\delta, b-\delta)$.

\begin{lemma}\label{l:variable_pos} 
  Let $n_0$ be as in Lemma~\ref{l:single_pos}.  Then there exist $\ve>0$ with
  $\mcal{U}_\ve(g) \subset \mcal{E}_0$ and $\kappa>0$ such that for all $f_i \in
  \mcal{U}_\ve(g)$, $\mcal{P}_{F_{n_0}}(\vp) \geqs \kappa$ for all $\vp \in \mcal{D}_{a^*}$.
\end{lemma}

\begin{proof}[Proof of Lemma~\ref{l:variable_pos}]
  Let $\vp \in \mcal{D}_{a^{*}}$ be fixed.  In the argument below, $\ve > 0$ and $\de > 0$
  will be taken to be as small as is needed ($\ve$ and $\de$ depend on $g$ and on $a^{*}$
  but not on $\vp$).  We let $n_{1}$, $J = J(\vp) \in \mcal{A}_{n_{1}} (g)$, $s(J) \in
  \mbb{Z}^{+}$, and $I \in \mcal{A}_{1} (g)$ be as in the proof of
  Lemma~\ref{l:single_pos}.  In particular, $\ve$ is small enough (depending on $g$ and
  $n_{1}$) so that $\Phi_{n} : \mcal{A}_{n} (g) \to \mcal{A} (F_{n})$ is well defined for
  all $f_{i} \in \mcal{U}_{\ve} (g)$ and the following $2$ values of $n$: $n=n_{1}$ and
  $n=N$, where $N=N(g)$ is the enveloping time for $g$.

  We claim that for every $I \in \mcal{A}_{1} (g)$ and $f_{i} \in \mcal{U}_{\ve} (g)$, we
  may assume that $F_{N} (I_{\de}) = \mcal{S}^{1}$.  For each $I' \in \mcal{A}_{N} (g) |
  I$, $g^{N} (I')$ and $F_{N} (\Phi_{N} (I'))$ can be made arbitrarily close.  This
  conclusion remains true if we shrink $I'$ by a small amount, \textit{i.e.}, $\de$ (we
  need only do this for the leftmost and rightmost $I' \in \mcal{A}_{N} (g) | I$).  The
  assertion follows from this and the enveloping property of $g$.

  Now let $f_{i} \in \mcal{U}_{\ve} (g)$ be fixed, and let $J' = \Phi_{n_{1}}
  (J)$. Assuming $\ve$ is chosen sufficiently small, $J'' = J' \cap J \neq \emptyset$, and
  $F_{s(J)} (J'') \supset I_{\de}$ where $\de$ is as in the previous paragraph.  Thus
  $F_{s(J)+N} (J'') = \mcal{S}^{1}$, and since $F_{n_{0}, s(J)+N+1}$ is onto, it follows
  that $F_{n_{0}} (J'') = \mcal{S}^{1}$. Noting that $\vp | J'' \geqs \frac{1}{2}$, we
  conclude that
\begin{equation*}
\mcal{P}_{F_{n_{0}}} (\vp) \geqs \frac{1}{2} (M_{0} + \ve)^{-n_{0}} \defas \ka.
\end{equation*}
\end{proof}

\noindent {\bfseries Step~\ref{step3}.} The matching process introduces, for $\vp \in
\mcal{D}_{a^*}$ with $\vp \geqs \kappa$, a new density $$\hat \vp \defas \frac{\vp-
  \kappa}{1- \kappa}.$$ (We may subtract off any amount $\leqs \kappa$, the only requirement
being that $\hat \vp$ remains $\geqs 0$.)  Since subtracting a constant does not diminish
variation, and magnifying it by a constant $c$ magnifies the variation by at most $c$, it
follows that $\hat \vp \in \mcal{D}_{a^*(1-\kappa)^{-1}}$.

\begin{proof}[Proof of Theorem~\ref{mt:pw_local}]
  We first iterate $\vp$ and $\psi$ until $\mcal{P}_{F_{n}} (\vp) \in \mcal{D}_{a^{*}}$
  and $\mcal{P}_{F_{n}} (\psi) \in \mcal{D}_{a^{*}}$.  This accounts for the prefactor
  $C_{(\vp,\psi)}$ in~\eqref{e:m_locres_est}.  We then follow the matching scheme in the
  proof of Theorem~\ref{mt:ue_maps}, obtaining $\Lambda =
  (1-\kappa)^{(n_0+\tau(a^*(1-\kappa)^{-1}))^{-1}}$.
\end{proof}

%%%%%%%%%%%%%%%%%%%%%%%%%%%%%%%%%%%%%%%%%%%

\subsection{Proof of global results}

Since $\gamma([a,b])$ is compact, we may assume it lies in a subset $\mcal{E}_0$ of $\mcal
E$ with uniformly bounded derivatives and a minimum expansion $\la_0>2$ as in
Section~\ref{ss:pf_loc_res}. This implies in particular that the set $\mcal{D}_{a^*}$ can
be taken to be uniform for all $g \in \gamma([a,b])$.

For each $g \in \mcal{E}_0$, there are three numbers that are relevant:
\begin{enumerate}[(1)]
\item $\ve(g)$, which describes the size
of the neighborhood in which our local results apply;
\item $\kappa(g)>0$ as given by Lemma~\ref{l:variable_pos};
\item
$n(g) \defas n_{0} (g) + \ta (a^{*} (1 - \ka)^{-1})$ where $n_{0}$ is as in
Lemma~\ref{l:single_pos} and involves the enveloping time of $g$.
\end{enumerate}
These quantities depend not just on the derivatives of $g$ but on its geometry, {\it i.e.}
how $\mcal{A}_n(g)$ partitions $\mcal{S}^1$, how quickly the covering property takes hold,
and so on.  Our local results imply that for all $f_i \in \mcal{U}_{\ve(g)}(g)$ and $\vp,
\psi \in \mcal{D}_{a^{*}}$, $n(g)$ is the number of steps at the end of which we are
guaranteed that the pair of densities has been matched once, and that their unmatched
parts, renormalized, are returned to $\mcal{D}_{a^*}$.  Moreover, the amount matched is $
\geqs \kappa(g)$.

\begin{proof}[Proof of Theorem~\ref{mt:pw_global}]
  For each $t \in [a,b]$, let $V_\alpha(t)$ denote the $\alpha$-neighborhood of $t$ in
  $\mathbb R$, and let $\alpha(t)>0$ be such that $\gamma(V_{\alpha(t)}(t) \cap [a,b])
  \subset \mcal{U}_{\ve(\gamma(t))}(\gamma(t))$.  By compactness, there exist $z_1< z_2 <
  \cdots < z_D$ such that $\bigcup_j V_{\frac12 \alpha(z_j)}(z_j)$ covers $[a,b]$. Let
  $g_j = \gamma(z_j)$, $V_j = V_{\alpha(z_j)}(z_j)$, and $\frac12 V_j = V_{\frac12
    \alpha(z_j)}(z_j)$.  Define $\delta_0 \defas \min_j \frac{\alpha(z_j)} {2n(g_j)}$.

  We claim that if $t_i$ defines a partition on $[a,b]$, the mesh $\Delta$ of which is
  $\leqs \delta_0$, then the $f_i=\gamma(t_i)$ will have the desired properties. Consider
  $f_i$ for arbitrary $i$, and let $\vp, \psi \in \mcal{D}_{a^*}$.  Since $t_i \in \frac12
  V_j$ for some $j$, our choice of $\delta_0$ assures that $f_i, f_{i+1}, \cdots,
  f_{i+n(g_{j})-1} \in \mcal{U}_{\ve(g_j)}(g_j)$. Thus a matching will take place, and the
  process can be repeated again at the end of $n(g_{j})$ steps. Since $\max_j n(g_j) <
  \infty$ and $\min_j \kappa (g_j)>0$, exponential loss of memory is proved.
\end{proof}

The argument above applies to $\gamma$ defined on a compact interval. If the curve in
$\mcal E$ is infinite, one simply divides it up into suitably short segments and treats
them one at a time (see Remark~\ref{rmk:variable}).

%%%%%%%%%%%%%%%%%%%%%%%%%%%%%%%%%

\bibliographystyle{amsplain}
\bibliography{lasota_yorke_submission}

\end{document}